\documentclass{amsart}

\input{epsf}
\usepackage{graphicx}
\usepackage{psfrag}
\usepackage{graphics}

\newtheorem{thm}{Theorem}[section]
\newtheorem{cor}[thm]{Corollary}
\newtheorem{lem}[thm]{Lemma}
\newtheorem{prop}[thm]{Proposition}
\newtheorem{claim}[thm]{Claim}
\theoremstyle{definition}

\theoremstyle{remark}
\newtheorem{rem}[thm]{Remark}
\numberwithin{equation}{section}

\begin{document}

\title[Convergence of general inverse $\sigma_k$-flow on K\"{a}hler manifolds with Calabi Ansatz]{Convergence of general inverse $\sigma_k$-flow on K\"{a}hler manifolds with Calabi Ansatz}%
\author{Hao Fang}%
\address{Hao Fang, 14 McLean Hall, Iowa City, IA 52242}%
\email{hao-fang@uiowa.edu}%
\author{Mijia Lai}
\address{Mijia Lai, 915 Hylan Building,
University of Rocheste, RC Box 270138, Rochester, NY 14627}%
\email{lai@math.rochester.edu}%

\begin{abstract}
We study the convergence behavior of the general inverse $\sigma_k$-flow on K\"{a}hler manifolds with initial metrics satisfying the Calabi Ansatz. The limiting metrics can be either smooth or singular. In the latter case, interesting conic singularities along negatively self-intersected sub-varieties are formed as a result of  partial blow-up.
\end{abstract}
\maketitle

\section{Introduction}

Geometric flows are powerful tools to study the metric, algebraic and topological properties of the underlying manifold. An important example is the Ricci flow introduced by Hamilton~\cite{H} three decades ago and its K\"{a}hlarian version, the K\"{a}hler-Ricci flow~\cite{Cao}. Since then, both have developed into significant research fields. See~\cite{T,SW4} for more complete surveys and further references on the K\"{a}hler-Ricci flow.

In~\cite{FLM, FL}, we have introduced the general inverse $\sigma_k$-flow on compact K\"{a}hler manifolds, which is a generalization of the $J$-flow~\cite{Chen,Do,SW1}. In this paper, we study some concrete examples to explore several applications of the general inverse $\sigma_k$-flow in algebraic geometry and fully non-linear partial differential equations.

We recall the definition of the general inverse $\sigma_k$-flow. Let $(X, \omega) $ be a compact K\"{a}hler manifolds of dimension $n$, with a fixed K\"{a}hler form $\omega$. Let $\chi$ be another K\"{a}hler form. For a fixed integer $k\in [1,n]$, let
\begin{align} \notag
c_k={n \choose k}\frac{\int_X \chi^{n-k}\wedge \omega^k}{\int_X \chi^n},
\end{align}
and define
\begin{align} \notag
\sigma_k(\chi)={ n \choose k} \frac{\chi^k \wedge \omega^{n-k}}{\omega^n}.
\end{align}
We are interested in $\sigma_k(\chi)$, which is a function defined  globally on $X$. Locally, it is the $k$-th elementary symmetric polynomial of eigenvalues of $\chi$ with respect to $\omega$,  $\{\lambda_{1},\cdots, \lambda_{n}\}$. In other words:
$$\sigma_k(\chi)=\sigma_k(\lambda_1, \cdots, \lambda_n)=\sum_{1\leq i_1 <i_2<\cdots <i_k\leq n} \lambda_{i_1} \lambda_{i_2} \cdots \lambda_{i_k}.$$

The general inverse $\sigma_k$-flow is defined for $$\chi_{\varphi}=\chi+\frac{\sqrt{-1}}{2}\partial{\bar\partial}\varphi\in\mathcal{P}_{\chi}=\{ \varphi\in C^{\infty}(X) | \chi_{\varphi}:=\chi+\frac{\sqrt{-1}}{2}\partial \bar{\partial} \varphi >0 \},$$
by
\begin{align} \label{inversek}
\left\{
\begin{array}{l l}
\frac{\partial}{\partial t} \varphi(x,t) &=
F(\frac{\sigma_{n-k}(\chi_{\varphi})}{\sigma_n(\chi_{\varphi})})-F(c_k),\\
\varphi(0)& =0,
\end{array}
\right.
\end{align}
where $x\in X,t\in[0,\infty)$, $F\in C^{\infty}(\mathbb{R}_{+}, \mathbb{R})$ satisfies the following conditions:
\begin{align} \label{concavity}
F'(x)<0, \ \quad \  F''(x)\geq 0, \ \quad F''(x)+\frac{F'(x)}{x}\leq 0.
\end{align}
Note that (\ref{concavity}) implies the ellipticity and strong concavity of the flow (\ref{inversek}). See~\cite{FL} for details.

Any stationary point of (\ref{inversek}) corresponds to a metric $\tilde{\chi}\in[\chi]$ satisfying the following K\"{a}hlerian inverse $\sigma_k$ equation on $X$:
\begin{align} \label{critical}
c_k \tilde{\chi}^n = {n\choose k} \tilde{\chi}^{n-k}\wedge \omega^k.
\end{align}
Locally, (\ref{critical}) can be written as
\begin{align} \label{local}
\frac{\sigma_{n-k}(\tilde{\chi})}{\sigma_n(\tilde{\chi})}=\sigma_k(\tilde{\chi}^{-1})=c_k,
\end{align}  where $\tilde{\chi}=\frac{\sqrt{-1}}{2}\tilde{\chi}_{i\bar{j}} dz_1\wedge d\bar{z_j}$ and $\tilde\chi^{-1}$ is the inverse of the matrix $(\tilde\chi_{i\bar{j}})$.

Equation (\ref{local}) leads to an obvious necessary condition for equation(\ref{critical}) to admit a smooth solution, which was first formulated in~\cite{SW1}  for the $J$-flow.
Define
\begin{align} \label{cone}
\mathcal{C}_k(\omega)=&\{ [\chi]>0\, |\,\exists \chi' \in [\chi], \\ \notag
&\text{such that}\, \  nc_k \chi'^{n-1}-{n \choose k}(n-k)\chi'^{n-k-1}\wedge \omega^k>0\, 
\}.
\end{align}
Note that for $k=n$, (\ref{cone}) holds for any K\"{a}hler class. Hence $\mathcal{C}_n(\omega)$ is the entire K\"{a}hler cone.
If there exists a smooth  metric $\tilde{\chi}\in[\chi]$ solving (\ref{critical}), it is necessary that
\begin{align} \label{condition}
[\chi]\in \mathcal{C}_k(\omega).
\end{align}

Following our earlier work with Ma~\cite{FLM}, we have shown that condition (\ref{condition}) is  also sufficient for the existence of smooth solutions of (\ref{local}) via the convergence of the flow (\ref{inversek})~\cite{FL}.

\begin{thm}
Let $(X,\omega,\chi)$ be given as above, then the flow (\ref{inversek}) has long time existence and converges to the critical metric $\tilde{\chi}$ satisfying (\ref{critical}) if and only if $[\chi]\in \mathcal{C}_k(\omega)$.
\end{thm}

It is easy to see that for $k=1$ and $F(x)=-x$, (\ref{inversek}) reduces to the $J$-flow. The smooth convergence of the $J$-flow has significant geometric implication on the properness of Mabuchi energy (see ~\cite{SW1}). Another special case of the flow occurs when $k=n$, where the critical equation is a complex Monge-Amp\`{e}re equation. If one takes $F(x)=-\log x$,  then the flow resembles K\"{a}hler-Ricci flow.

Same as the $J$-flow, the general inverse $\sigma_k$-flow (\ref{inversek}) always has long time existence (see ~\cite{FL}). It is thus interesting to study  convergence properties of the flow when the condition (\ref{condition}) fails to hold. While convergence to smooth metrics is no longer expected, due to the geometric set-up, blow-up behavior of the solution along proper subvarieties is expected (cf.~\cite{SW1}). In particular, it is our hope that the analytical behavior of the limit metric will reflect the algebro-geometric properties of the original K\"{a}hler manifold $X$.

From an analytical point of view, (\ref{critical}) deserves study in its own right. For $k=n$, it is a complex Monge-Amp\`{e}re equation. If $[\chi]$ is K\"{a}hler, by Yau's renowned solution of Calabi conjecture~\cite{Y}, (\ref{critical}) admits a smooth solution unique up to a constant. If $[\chi]$ lies on the boundary of K\"{a}hler cone, i.e., $[\chi]$ is nef not K\"{a}hler, then (\ref{critical}) becomes a degenerate complex Monge-Amp\`{e}re equation. It is a subject of intensive study over past two decades following the pioneering work of Ko{\l}odziej~\cite{K}. When $\chi$ is a big semi-positive form, the result in~\cite{EGZ} implies that (\ref{critical}) admits a bounded plurisubharmonic solution. Such solution of degenerate complex Monge-Amp\`{e}re equation is used to produce singular K\"{a}hler-Einstein metrics on K\"{a}hler manifolds with indefinite anticanonical class ~\cite{EGZ}. There have been K\"{a}hler-Ricci flow approaches for more general canonical singular K\"{a}hler-Einstein metrics. See~\cite{TZ, ST1, ST2} for details and further references.

For $k\neq n$, we refer (\ref{critical}) as the K\"{a}hlerian inverse $\sigma_k$ equation. When $[\chi]$ lies on the boundary of $\mathcal{C}_k(\omega)$, this Monge-Amp\`{e}re type equation degenerates in an intriguing way. Suggested by the convergence results we obtained on general inverse $\sigma_k$ flow, we conjecture an analogous result of ~\cite{EGZ} on boundedness of the solution in pluri-potential sense still holds.

In this paper, we shall study the general inverse $\sigma_k$-flow assuming certain symmetry of initial data. This is partially inspired by similar results on the K\"{a}hler-Ricci flow~\cite{SW2,SW3,SY} and on K\"{a}hler-Ricci solitons~\cite{L}.  It is interesting to compare convergence behaviors of the general inverse $\sigma_k$-flow with those of the K\"{a}hler-Ricci flow.

\begin{thm}[\textbf{Main Theorem 1}]\label{main1}
Let $X=\mathbb{P}^n \#\overline{\mathbb{P}^n}$ be $\mathbb{P}^n$ blowing up at one point. Let $E_0$ and $E_{\infty}$ be the exceptional divisor and the pull-back of the divisor associated to  $\mathcal{O}_{\mathbb{P}^n}(1)$ respectively. Assume that $\chi$, $\omega$ are  K\"{a}hler metrics on $M$ satisfying Calabi Ansatz (see Section 2 for details) such that
\[
\omega \in \alpha[E_{\infty}]-[E_0],\quad \chi \in \beta[E_{\infty}]-[E_0].
\]
Let $\chi_t$ be the solution of the flow (\ref{inversek}), then the following convergence behavior of $\chi_t$ holds:
\begin{enumerate}
  \item If $\frac{\alpha^k \beta^{n-k}-1}{\beta^n-1} >\frac{n-k}{n}$, then as $t\to\infty$, $\chi_t\xrightarrow{C^{\infty}(X)}\chi_{\infty}$,  a smooth K\"{a}hler metric satisfying (\ref{critical}).
  \item If $\frac{\alpha^k \beta^{n-k}-1}{\beta^n-1} =\frac{n-k}{n}$, then as $t\to\infty$, $\chi_t\xrightarrow{C^{\infty}(X\setminus E_0)} \chi_{\infty}$,  a singular K\"{a}hler metric that is smooth
 away from $E_{0}$ and has  conic singularity at $E_0$ of angle $\pi$. Further more, there is a universal constant $C$ such that the oscillation of the limiting potential $\varphi_{\infty}$ satisfies $$\text{osc}\  \varphi_{\infty}\leq C.$$
   \item If $\frac{\alpha^k \beta^{n-k}-1}{\beta^n-1} <\frac{n-k}{n}$, then as $t\to\infty$, $\chi_t\xrightarrow{C^{\infty}(X\setminus E_0)} \chi_{\infty}+(\lambda-1)[E_0]$, a K\"{a}hler current. Here $\lambda\in(1,\beta)$ is unique such that
\begin{align} \label{lambda}
(n-k)(\frac{\beta}{\lambda})^{k}+k(\frac{\lambda}{\beta})^{n-k}=n\alpha^{k},
\end{align}
and $[E_0]$ is the current of integration along the exceptional divisor $E_0$. As is in case (2),  $\chi_{\infty}\in \beta[E_{\infty}]-\lambda[E_0]$ is also a singular K\"{a}hler metric with conic singularity with angle $\pi$ transverse to $E_0$.
\end{enumerate}
\end{thm}

\begin{rem}\label{remark1} Note that in the case (3) of Theorem~\ref{main1}, $\chi_{\infty}$ can be also obtained as  the limit of flow (\ref{inversek}) from some smooth initial data  $(X, \omega, \chi)$ where $\chi \in \beta[E_{\infty}]-\lambda[E_0]$ and $\omega\in \alpha[E_{\infty}]-[E_0]$. This is an interesting example of partial blow-up in the sense of algebraic geometry using analytical tools. See~\cite{SW2, SW3, SY} for some corresponding results for the K\"{a}hler-Ricci flow.
\end{rem}

\begin{rem}
For the $J$-flow on K\"{a}hler surface ($k=1, n=2$),  the cone condition (\ref{cone}) reduces to a simple class condition:
\begin{align} \label{1.5}
[c_1\chi-\omega]>0.
\end{align}
Donaldson ~\cite{Do} noticed such condition holds for all K\"{a}hler classes $[\chi]$ and $[\omega]$ if there are no curves of negative self-intersection. He also conjectured if (\ref{1.5}) fails to hold then one might expect the flow to blow up over some such curves. This was confirmed in ~\cite{SW1} in the sense that the quantity $|\varphi|+|\Delta_{\omega}\varphi|$ blows up, where $\varphi$ is the solution of the $J$-flow. In fact, their estimate is on any dimension $n$ when the condition (\ref{cone}) is violated. They posed a question on improving these estimates. Our results give a partial answer to their question. In particular, case (2) of the theorem asserts that for the given $X$ with initial metrics satisfying Calabi ansatz, $|\varphi|$ stays bounded.
\end{rem}

Figure~\ref{fig1} illustrates different cases of Theorem~\ref{main1}. In Figure~\ref{fig1}, the K\"{a}hler cone $\mathcal{K}(X)$ is given by the region $\{ b[E_{\infty}]-a[E_0]| b>a>0\}$, i.e., the upper half quadrant; while $\mathcal{C}_k(\omega)$ defined by (\ref{cone}) is the cone between two dotted lines.  $\beta_0$ is the constant satisfying $\frac{\beta_0^{n-k} \alpha^k-1}{\beta_0^n-1}=\frac{n-k}{n}$, which defines the boundary of $\mathcal{C}_{k}(\omega)$ inside of $\mathcal{K}(X)$. Therefore case (2) of Theorem~\ref{main1} corresponds to $[\chi]$ lying on the upper boundary of $\mathcal{C}_{k}(\omega)$; case (3) of Theorem~\ref{main1} corresponds to $[\chi]\in \mathcal{K}(X)\setminus \overline{\mathcal{C}_{k}(\omega)}$. The limit $\chi_{\infty}$ of case (3) jumps to the class $\beta[E_{\infty}]-\lambda[E_0]$ which lies on the upper boundary of $\mathcal{C}_k(\omega)$. The relation (\ref{lambda}) then follows.

\begin{figure}[htpb!]\label{fig1}
\begin{center}

\psfrag{9}{$-[E_0]$}
\psfrag{3}{$\alpha$}
\psfrag{4}{$\beta_0$}
\psfrag{5}{$\beta$}
\psfrag{10}{$[E_\infty]$}
\psfrag{8}{$\mathcal C_k(\omega)$}
\psfrag{1}{$1$}
\psfrag{2}{$\lambda$}
\psfrag{7}{$[\omega]$}
\psfrag{6}{$[\chi]$}
\psfrag{11}{$[\chi_{\infty}]$}
\includegraphics[height=60mm]{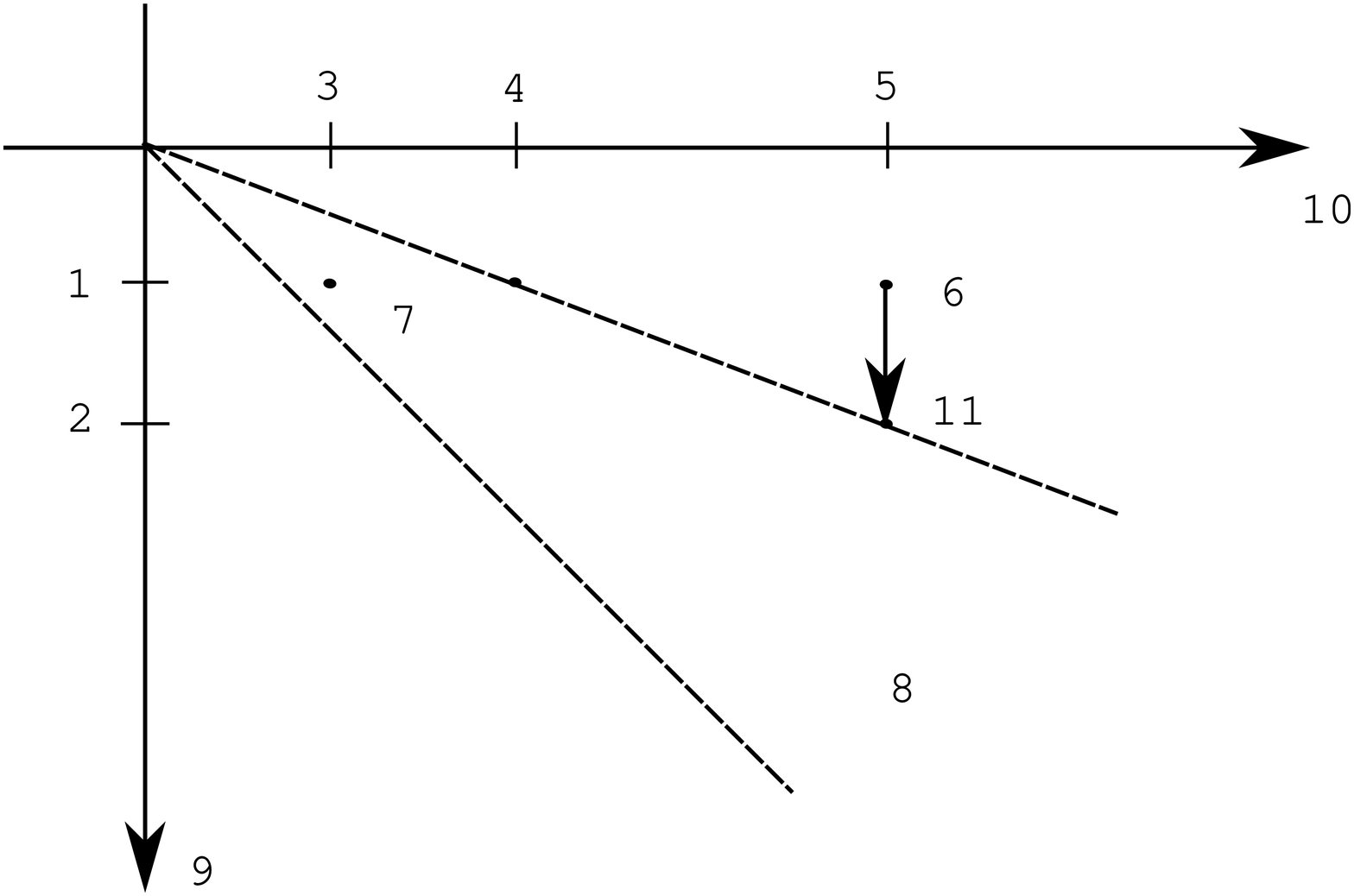}

\caption{}\label{graph1}
\end{center}
\end{figure}

Note the constant $\frac{\alpha^k \beta^{n-k}-1}{\beta^n-1}$ in the Theorem~\ref{main1} is in fact the topological constant
\[
\frac{\int_X \chi^{n-k}\wedge \omega^k}{\int_X \chi^n}=\frac{[\chi]^{n-k}[\omega]^k}{[\chi]^n}.
\]
Hence by Theorems 1.1 and 1.2, we obtain

\begin{cor} Fix the notation as in Theorem~\ref{main1}. Let $\omega \in \alpha[E_{\infty}]-[E_0] $ and  $\chi \in \beta [E_{\infty}]- [E_0]$. Assume that $\omega$ satisfies the Calabi Ansatz. Then $[\chi] \in \mathcal{C}_{k}(\omega)$ if and only if $\frac{\int_X \chi^{n-k} \wedge \omega^k}{\int_X \chi^{n}}=\frac{\alpha^k \beta^{n-k}-1}{\beta^n-1} > \frac{n-k}{n}$. In particular, $\mathcal{C}_k(\omega)$ is a convex open cone.
\end{cor}

In order to get examples involving singularities of higher co-dimension, we study the flow (\ref{inversek}) on more complicated manifolds admitting the Calabi Ansatz.

Let $X_{m,n}=\mathbb{P}(\mathcal{O}_{\mathbb{P}^n}\oplus \mathcal{O}_{\mathbb{P}^n}(-1)^{\oplus (m+1)})$ be a projective bundle over $\mathbb{P}^n$ of total dimension $m+n+1$. Let $D_{\infty}$ be the divisor given by $\mathbb{P}(\mathcal{O}_{\mathbb{P}^n}(-1)^{\oplus (m+1)})$ and $D_{H}$ be the pullback of the divisor on $\mathbb{P}^n$ associated to $\mathcal{O}_{\mathbb{P}^n}(1)$.

For future use, let $\tilde{X}_{m,n}$ be the blow up of $X_{m,n}$ along $P_0$, where $P_0\subset X_{m,n}$ is the projectivization of the section $(1,0,\cdots, 0)\in \mathcal{O}_{\mathbb{P}^n}\oplus \mathcal{O}_{\mathbb{P}^n}(-1)^{\oplus (m+1)}$. Note that $P_{0}$ is of dimension $n$. We denote the resulting exceptional divisor in  $\tilde{X}_{m,n}$ by $E$.

\begin{thm}[\textbf{Main Theorem 2}]\label{main2}
Let $X_{m,n}$ be as given as above. Assume that $\omega, \chi$ are two K\"{a}hler metrics satisfying the Calabi Ansatz and
\[
\omega \in [D_H]+b[D_{\infty}]\quad \text{and} \quad \chi\in [D_H]+b'[D_{\infty}],
\]where $[D_{H}]$ and $[D_{\infty}]$ are two generators of the divisor group of $X_{m,n}$.
Let $\chi_t$ be the solution of the general inverse $\sigma_k$-flow (\ref{inversek}), then we have the following:
\begin{description}
  \item[ $k>n$] $\chi_t\xrightarrow{C^{\infty}(M)}\chi_{\infty}$,  a smooth K\"{a}hler metric solving (\ref{critical}), as $t\to\infty$.\\
  \item[ $k\leq n$] Recall $c_k=\frac{{m+n+1 \choose k} \int_{X_{m,n}} \chi^{m+n+1-k}\wedge \omega^k}{\int_{X_{m,n}}\chi^{m+n+1}}$, we have:                \begin{enumerate}
                   \item $c_k>{n \choose k}$: $\chi_t\xrightarrow{C^{\infty}(M)}\chi_{\infty}$,  a smooth K\"{a}hler metric solving (\ref{critical}), as $t\to\infty$.
                    \item $c_k={n \choose k}$: $\chi_t\xrightarrow{C^{\infty}_{loc}(X_{m,n}\setminus P_0)} \chi_{\infty}$, a singular K\"{a}hler metric with cone singularity of angle $\pi$ transverse to $P_0$ as $t\to\infty$; there is a universal contant $C$ such that $$\text{osc}\ \varphi_{\infty}\leq C.$$
                   \item $c_k<{n \choose k}$: Let $\pi$ be the blown up map: $\pi: \tilde{X}_{m,n} \to X_{m,n}$. Then $\pi^{*} (\chi_t)\xrightarrow{C^{\infty}_{loc}(X_{m,n}\setminus P_0)} \chi_{\infty}+\lambda[E]$, as $t\to\infty$. Here $[E]$ is the current of integration along $E$; $\lambda \in (0, b')$ is a constant uniquely determined by the geometric data; $\chi_{\infty}\in [D_{H}]+b'[D_{\infty}]-\lambda[E]$ is a singular K\"{a}hler metric on $\tilde{X}_{m,n}$ with cone singularity of angle $\pi$ transverse to fibre direction of $E$ such that on ${\tilde X}_{m,n}\setminus E$ \[
                       \alpha \chi_{\infty}^{m+n+1}=\chi_{\infty}^{m+n+1-k}\wedge \pi^{*}(\omega)^k,
                       \]
where
                       \[
                       \alpha=\frac{([D_{H}]+b'[D_{\infty}]-\lambda[E])^{m+n+1-k}([D_{H}]+b'[D_{\infty}])^k}{([D_{H}]+b'[D_{\infty}]-\lambda[E])^{m+n+1}}.\]                   \end{enumerate}
\end{description}

\end{thm}

Note that in the case (3) above, when no confusion arises, $D_{H}$ and $D_{\infty}$ are also used to denote the corresponding pulled-back divisors on $\tilde{X}_{m,n}$.

Combining Theorem 1.1 and Theorem~\ref{main2}, we also obtain

\begin{cor}
On $X_{m,n}$, assume that $\omega\in [D_H]+b[D_{\infty}]$ and $\chi\in [D_H]+b'[D_{\infty}]$. If $\omega$ satisfies the Calabi Ansatz, then
$\mathcal{C}_k(\omega)$ is the entire K\"{a}hler cone whenever $k>n$; and when $k\leq n$, $[\chi]\in \mathcal{C}_k(\omega)$ if and only if $c_k=\frac{{m+n+1 \choose k} \int_{X_{m,n}} \chi^{m+n+1-k}\wedge \omega^k}{\int_{X_{m,n}}\chi^{m+n+1}}>{n\choose k}$.
\end{cor}

 \begin{rem}
For both case (2) of Theorems~\ref{main1} and case (2) of Theorem~\ref{main2}, the critical equation (\ref{critical}) admits a bounded solution in the sense of pluri-potential theory.
This indicates a general result should hold for this type of equations. Thus, the results of ~\cite{K} and ~\cite{EGZ} on solutions for degenerate complex Monge-Amp\`{e}re equations may be extended to more general complex Monge-Amp\`{e}re type equations. We would like to discuss this aspect in future works.
\end{rem}

\begin{rem}
Similar to Remark~\ref{remark1}, it is easy to construct a proper inverse $\sigma_{k}$ type of flow on $\tilde{X}_{m,n}$ such that the limiting metric under the flow coincides with that of case (3) listed above. The inverse $\sigma_{k}$ flow can thus be viewed as an analytical method to connect birationally equivalent varieties.
\end{rem}

In this paper we only discuss some special K\"{a}hler manifolds with metrics satisfying strong symmetric conditions. However, the algebraic pictures revealed indicate that the geometric flows that we have studied can be used to transform between (possibly singular) algebraic varieties. In  subsequent works, we will discuss  general cases and their further applications to birational geometry.

The rest of paper is organized as follows. In Section 2, we study the $J$-flow on $\mathbb{P}^n \# \overline{\mathbb{P}^n}$ as a prototype. In Section 3, we use the same method to treat the general inverse $\sigma_k$-flow on $\mathbb{P}^n \# \overline{\mathbb{P}^n}$. In Section 4, we study the convergence behavior on $X_{m,n}$, a class of more generality.

{Acknowledgments: Both authors would like to thank Ben Weinkove, Jian Song and Lihe Wang for helpful discussion.}

\section{$J$-flow on $\mathbb{P}^n\# \overline{\mathbb{P}^n}$}

Let $X=\mathbb{P}^n \#\overline{\mathbb{P}^n}$ be $\mathbb{P}^n$ blowing up at one point. Denote $[E_0]$ and $[E_{\infty}]$ the exceptional divisor and  pull-back of the hyper-plane in $\mathbb{P}^{n}$
respectively. We have $H^{1,1}(X, \mathbb{R})= \text{span} \{ [E_0], [E_{\infty}]\}$ and any K\"{a}hler class $\Omega$ on $X$ is of the form:
\[
\Omega=b[E_{\infty}]-a[E_0], \quad  b>a>0.
\]

We recall Calabi Ansatz~\cite{Ca} on construction of rotational symmetric K\"{a}hler metrics on $X$. For notational convenience, let $d=\partial+\bar{\partial}$ and $d^{c}=\frac{i}{4\pi}(\partial-\bar{\partial})$, then $dd^{c}=\frac{i}{2\pi} \partial \bar{\partial}$. On $X \setminus (E_0\cup E_{\infty}) \cong \mathbb{C}^n \setminus 0$, we can associate a coordinate system $(z_{0},\cdots, z_{n})$. Define\[
\rho= \ln (|z_1|^2+|z_2|^2+ \cdots +|z_n|^2).\]

For a function $u\in C^{\infty}(\mathbb{R})$ such that

\[
  u'(\rho)>0, \quad  u''(\rho)>0,
\]  the (1,1)-form $\omega =d d^{c} u(\rho)$ is then a K\"{a}hler form.

In order to  extend $\omega$  to a smooth K\"{a}hler metric on $X$,  the following asymptotic properties of $u$ are required:
\begin{align} \label{asymptotic}
  (\dag). \quad &u_0(r):= u(\ln r)-a \ln r \\ \notag
       &\text{is extendable by continuity to a smooth function at} \, r=0, \, \text{and} \, u_0'(0)>0.\\ \notag
  (\ddag). \quad &u_{\infty}(r): =u(-\ln r) +b \ln r\\ \notag
        &\text{ is extendable by continuity to a smooth function at} \, r=0, \,  \text{and} \,  u_{\infty}'(0)>0.
\end{align}
It is easy to see, by the asymptotic behavior of $u$, that
\[
\lim_{\rho\to -\infty}u'(\rho)=a, \quad \lim_{\rho \to \infty} u'(\rho)=b.
\]
Moreover, since $u''(\rho)>0$,  $b>a$.  $a$ and $b$ characterize the K\"{a}hler class of $\omega$ in the following manner:
\[
 \omega \in b[E_{\infty}]-a[E_0].
\]

In this section, we treat the $J$-flow, which is a special case of general inverse $\sigma_{k}$ flows. It is defined as follows:

\begin{align} \label{jflow}
\left\{
\begin{array}{l l}
\frac{\partial}{\partial t} \varphi &=
c_1-\frac{n\chi_{\varphi}^{n-1}\wedge \omega}{\chi_{\varphi}^n},\\
\varphi(0)& =0.
\end{array}
\right.
\end{align}

After normalization we may assume
\[
\omega \in \alpha [E_{\infty}]-[E_0], \quad  \chi \in \beta[E_{\infty}]-[E_0];\quad   \alpha>1 , \quad \beta>1.
\]
If both $\omega$ and $\chi$ satisfy the Calabi Ansatz,  on the coordinate patch $X\setminus (E_0\cup E_{\infty})$, we have smooth functions $u,v\in C^{\infty}({\mathbb{R}})$ such that
\[
\omega=d d^{c} u(\rho),  \quad  \chi=d d^{c} v(\rho).
\]
This leads to
\[
\omega=(\frac{u'}{e^{\rho}}\delta_{ij}+ (u''-u')\frac{\bar{z_{i}}z_j}{e^{2\rho}} ) dz_i d\bar{z_{j}},\]
\[
\chi=(\frac{v'}{e^{\rho}}\delta_{ij}+ (v''-v')\frac{\bar{z_{i}}z_j}{e^{2\rho}} ) dz_i d\bar{z_{j}}.
\]
Thus the eigenvalues of $\chi$ with respect to $\omega$ are
\[
\underbrace{\frac{v'}{u'}, \frac{v'}{u'}, \cdots, \frac{v'}{u'}}_{(n-1)-times}\quad   \frac{v''}{u''}.
\]

It is easy to see that the $J$-flow preserves the Calabi Ansatz condition. Hence, we may assume that solution of the flow (\ref{jflow}) is $\chi_{\varphi(\cdot, t)}= dd^{c} v(\rho, t)$.
Consequently, (\ref{jflow}) is reduced to an evolution equation on $v(\rho, t)$:
\begin{align} \label{1.1}
\left\{
\begin{array}{l l}
\frac {\partial v}{\partial t} &=c_1- (n-1) \frac{u'}{v'}-\frac{u''}{v''},\\
v(\rho,0)&=v_{0}(\rho).
 \end{array}
\right.
\end{align}

The corresponding critical equation (\ref{critical}) is
\begin{align} \label{1.6}
c_1 =(n-1)\frac{u'}{v'}+\frac{u''}{v''}.
\end{align}

Taking one time derivative on (\ref{inversek}), and applying maximum principle, we obtain the bound
\begin{align}  \label{1.8}
0<C_1\leq \frac{\sigma_{n-1}(\chi_{\varphi})}{\sigma_n(\chi_{\varphi})}\leq C_2,
\end{align} for two universal constants $C_1$ and $C_2$ depending only on the initial data.

In terms of potential $u$ and $v$, (\ref{1.8}) is
\begin{align}
 \label{bound}
0<C_1 \leq (n-1)\frac{u'}{v'}+\frac{u''}{v''}\leq C_2.
\end{align}

To study flow (\ref{1.1}), we regard $(v'(\rho, t), u'(\rho))$ as a family of parametric plane curves. Since $u''>0$ and $v''>0$, each curve is the graph of a strict monotone increasing function $f(x,t)$, i.e.,
\begin{align} \label{cartesion}
f(v'(\rho,t),t)=u'(\rho).
\end{align}

We are concerned with the corresponding evolution equation for $f$.

\begin{prop}
The function $f=f(x,t)$ defined by (\ref{cartesion}) is a classical solution of the initial-boundary value problem
\begin{align} \label{evolutionf}
\frac{\partial f}{\partial t}= Q(f) [\frac{\partial^2 f}{\partial x^2}+(n-1)\frac{\frac{\partial f}{\partial x}}{x}-(n-1)\frac{f}{x^2}],
\end{align}
with
\[
f(1,t)=1,\quad  f(\beta, t)=\alpha \quad \forall t,
\]
where the initial value $f(x,0)=f_0(x)$ is any smooth monotone function satisfying the above boundary condition; $Q\in C^{\infty}([1,\alpha],\mathbb{R}_{\geq 0})$ is uniquely determined by  $u$ such that $Q(y)=0$ if and only if $y=1$ and $y=\alpha$.
\end{prop}

\begin{proof}
Taking one time derivative of (\ref{cartesion}), we get
\begin{align} \label{1.9}
\frac{\partial f}{\partial x} \frac{\partial v'}{\partial t} +\frac{\partial f}{\partial t}=0.
\end{align}
Taking one spacial derivative of (\ref{1.1}), we get
\begin{align} \label{1.2}
\frac{\partial v'}{\partial t} =-(n-1)(\frac{u'}{v'})'-(\frac{u''}{v''})'.
\end{align}
Taking spacial derivatives of (\ref{cartesion}) we also have
\begin{align} \label{1.3}
\frac{\partial f}{\partial x}=\frac{u''}{v''},\quad \frac{\partial^2 f}{\partial x^2}=\frac{(\frac{u''}{v''})'}{v''}.
\end{align}
Plugging (\ref{1.2}) and (\ref{1.3}) in (\ref{1.9}), we get
\begin{align}  \label{1.4}
\frac{\partial f}{\partial t}&=-\frac{\partial f}{\partial x}\frac{\partial v'}{\partial t}\\ \notag
                             &=\frac{\partial f}{\partial x}[(n-1)(\frac{u''}{v'}-\frac{u'v''}{v'^2})+(\frac{u''}{v''})']\\ \notag
                             &=\frac{\partial f}{\partial x}[(n-1)(\frac{u''}{v'}-\frac{u'v''}{v'^2})+\frac{\partial^2 f}{\partial x^2}v'']\\ \notag
                             &=u''[\frac{\partial^2 f}{\partial x^2}+(n-1)\frac{\frac{\partial f}{\partial x}}{x}-(n-1)\frac{f}{x^2}].
\end{align}
For a given $u$, since $u''>0$, then $u'$ is monotone, it follows that we can write $u''$ as
\[
u''=u''(u'^{-1}(f)):= Q(f).
\]
By the asymptotic behavior of $u'$ (\ref{asymptotic}), we have $Q(1)=Q(\alpha)=0$ and $Q(f)>0$ whenever $1<f<\alpha$.

The boundary and initial conditions follow directly from the limit behavior of $u'$ and $v'$.
\end{proof}

To study the convergence behavior of (\ref{1.1}), it suffices to study the convergence behavior of (\ref{evolutionf}). This is a degenerate parabolic equation of one spacial dimension. A priori, we know the long time existence and we also have a uniform $C^1$ bound on $f$ by (\ref{bound}), i.e.,
\[
\frac{\partial f}{\partial x} =\frac{u''}{v''}\leq C_2.
\]
Therefore we have a uniform limit
\begin{align} \label{limit}
\lim_{t\to \infty} f(x,t) =f_{\infty}(x).
\end{align}
It follows that $f_{\infty}(x)$ is a  weak solution of the corresponding stationary problem:
\begin{align} \label{1.11}
Q(f) [\frac{\partial^2 f}{\partial x^2}+(n-1)\frac{\frac{\partial f}{\partial x}}{x}-(n-1)\frac{f}{x^2}]=0, \quad f(1)=1, \, f(\beta)=\alpha.
\end{align}

Equivalently, we can write (\ref{1.11}) as
\begin{align}  \label{1.12}
\Psi_{\{1<f<\alpha\}}(\frac{\partial^2 f}{\partial x^2}+(n-1)\frac{\frac{\partial f}{\partial x}}{x}-(n-1)\frac{f}{x^2})=0
\end{align}
subject to the boundary condition $f(1)=1$ and $f(\beta)=\alpha$, where $\Psi_{.}$ is the characteristic function of a set. Standard elliptic theory shows that $f_{\infty}$ is a strong solution of (\ref{1.11}) and (\ref{1.12}).

It is straightforward to calculate  all possible solutions of (\ref{1.12}). Piece-wisely, they are either  constant functions $f=1$ and $f=\alpha$, or solutions of the differential equation
\begin{align} \label{ode1}
f''+(n-1)\frac{f'}{x}-(n-1)\frac{f}{x^2}=0,
\end{align}
which can be written as $Ax+{B\over x^{n-1}}$, for appropriate constants $A$ and $B$.

Since $f(x,t)$ is monotone for all $t$, limit $f_{\infty}$ can only be of the form

\begin{align} \label{1.13}
f_{\infty}(x)=\begin{cases}
1, & 1\leq x\leq s, \\
g(x), & s \leq x\leq t,\\
\alpha, & t\leq x \leq \beta,
\end{cases}
\end{align}
where $g$ is the solution of (\ref{ode1}) on $[s,t]$.

Due to the degeneracy condition on $Q$, $f_{\infty}$ may lose regularity at $x=s$ and $x=t$. Nevertheless, on any compact subset in which $\{1<f_{\infty}(x)<\alpha\}$, (\ref{evolutionf}) is uniform elliptic, and we obtain higher order estimates (cf. ~\cite{W}) from general theory of nonlinear parabolic equations, consequently the convergence (\ref{limit}) is smooth in region $\{x| 1<f_{\infty}(x)<\alpha\}$.

Consider the ODE
\begin{align} \label{ode}
f''(x)+(n-1)\frac{f'(x)}{x}-(n-1)\frac{f(x)}{x^2}=0
\end{align}
with boundary value $ f(1)=1$ and $ f(\beta)=\alpha$,
it has a unique solution:
 \begin{align} \label{finfty}
\tilde{f}(x)=ax+\frac{b}{x^{n-1}},
\end{align}with $a=\frac{\alpha\beta^{n-1}-1}{\beta^{n}-1}$ and $b=\frac{\beta^n-\alpha\beta^{n-1}}{\beta^n-1}$

We now have the following

\begin{lem}\label{lemma1.5}
Let $\varphi(x,t), \psi(x,t)\in\mathcal{P}_{\chi}$ be two solutions of the flow (\ref{inversek}) with initial values $\varphi_0$ and $\psi_0$ respectively.
Let $f_{t},{\bar f}_{t}$ be their corresponding functions satisfying (\ref{evolutionf}). Then there exists a universal constant $C$ such that
\begin{align}\label{new1}
|\varphi(x,t)-\psi(x,t)|_{C^0}\leq C, \quad \forall t.
\end{align}
In particular, for ${f}_{\infty}(x)=\lim_{t\to\infty}f_{t}(x)$ and $\tilde{f}_{\infty}(x)=\lim_{t\to\infty}{\bar f}_{t}(x),$ we have
$${f}_{\infty}(x)=\tilde{f}_{\infty}(x).$$
\end{lem}

\begin{proof}  Taking difference of flows with two initial values, we get $\varphi(x,t)-\psi(x,t)$ satisfies a parabolic equation
\[
\frac{\partial [\varphi(x,t)-\psi(x,t)]}{\partial t}=F^{i\bar{j}}((1-s_t)\chi_{\varphi(t)}+s_{t}\chi_{\psi(t)}) (\varphi(t)-\psi(t))_{i\bar{j}},
\]
where $0<s_{t}<1$.
Inequality (\ref{new1}) then follows from the maximum principle.

For the second part of the lemma,  we consider the corresponding limits for $v'(\rho,t)$, denoted by $v'_{\infty}$ and $\tilde{v}'_{\infty}$ respectively.  By (\ref{1.13}), there exist constants $s,t,\tilde{s},\tilde{t}$ such that
\[
v'_{\infty}(-\infty)=s, \quad v'_{\infty}(\infty)=t,
\]
and
\[
\tilde{v}'_{\infty}(-\infty)={\tilde s}, \quad \tilde{v}'_{\infty}(\infty)={\tilde t}.
\]
Thus, $|v_{\infty}-\tilde{v}_{\infty}|$ is not uniformly bounded unless $s=\tilde s$ and $t=\tilde t$. We have thus proved the lemma.
\end{proof}

We may now state the following key result of this section:

\begin{thm}\label{111} The flow (\ref{evolutionf}) converges to a unique limit $f_{\infty}(x)$. For the expression of $f_{\infty}$, we have following four cases:
\begin{enumerate}
  \item $\alpha> \beta$: $f_{\infty}(x)=\tilde{f}(x)$;
  \item $\alpha<\beta$ and $\frac{\alpha\beta^{n-1}-1}{\beta^n-1}>\frac{n-1}{n}$: $f_{\infty}(x)=\tilde{f}(x)$;
  \item $\alpha<\beta$ and $\frac{\alpha\beta^{n-1}-1}{\beta^n-1}=\frac{n-1}{n}$: $f_{\infty}(x)=\tilde{f}(x)$;
  \item $\alpha<\beta$ and $\frac{\alpha\beta^{n-1}-1}{\beta^n-1}<\frac{n-1}{n}$:
  Let
  \begin{align} \notag
  \lambda=\inf \{ \lambda'\, | \, \exists g(x) \, \text{satisfies} \, (\ref{ode}), \,
  g(\lambda')=1, \, g(\beta)=\alpha\, \text{and}  \, \, g\geq 1\}.
  \end{align}
Let $g$ be the corresponding solution of (\ref{ode}) with $g(\lambda)=1$ and $g(\beta)=\alpha$, we have in this case
  \[
f_{\infty}(x) =
\begin{cases}
1, & 1\leq x\leq \lambda; \\
g(x), & \lambda \leq x\leq \beta.
\end{cases}
\]

\end{enumerate}
\end{thm}

To better understand the theorem, we illustrate initial and limit functions of four cases by Figure~\ref{figure2}.
\begin{figure}[h!]
\begin{center}
\psfrag{2}{$\alpha$}
\psfrag{4}{$\beta$}
\psfrag{5}{$f_{_0}$}
\psfrag{6}{$\tilde{f}$}
\psfrag{case 1}{Case 1: $\alpha>\beta$}
\psfrag{8}{$\alpha$}
\psfrag{10}{$\beta$}
\psfrag{12}{$f_{_0}$}
\psfrag{11}{$\tilde{f}$}
\psfrag{case 3}{Case 2: $\alpha<\beta$, $\tilde{f}'(1)>0$}
\psfrag{19}{$\alpha$}
\psfrag{17}{$\beta$}
\psfrag{20}{$f_{_0}$}
\psfrag{21}{$\tilde{f}$}
\psfrag{case 2}{Case 3: $\alpha>\beta$, $\tilde{f}'(1)=0$}
\psfrag{case 4}{Case 4: $\alpha>\beta$, $\tilde{f}'(1)<0$}
\psfrag{23}{$\alpha$}
\psfrag{25}{$\beta$}
\psfrag{28}{$\tilde{f}$}
\psfrag{26}{$f_{_0}$}
\psfrag{27}{$f_{\infty}$}
\psfrag{1}{1}
\psfrag{3}{1}
\psfrag{7}{1}
\psfrag{9}{1}
\psfrag{13}{1}
\psfrag{14}{1}
\psfrag{15}{1}
\psfrag{16}{1}
\psfrag{29}{$\lambda$}
\includegraphics[height=80mm]{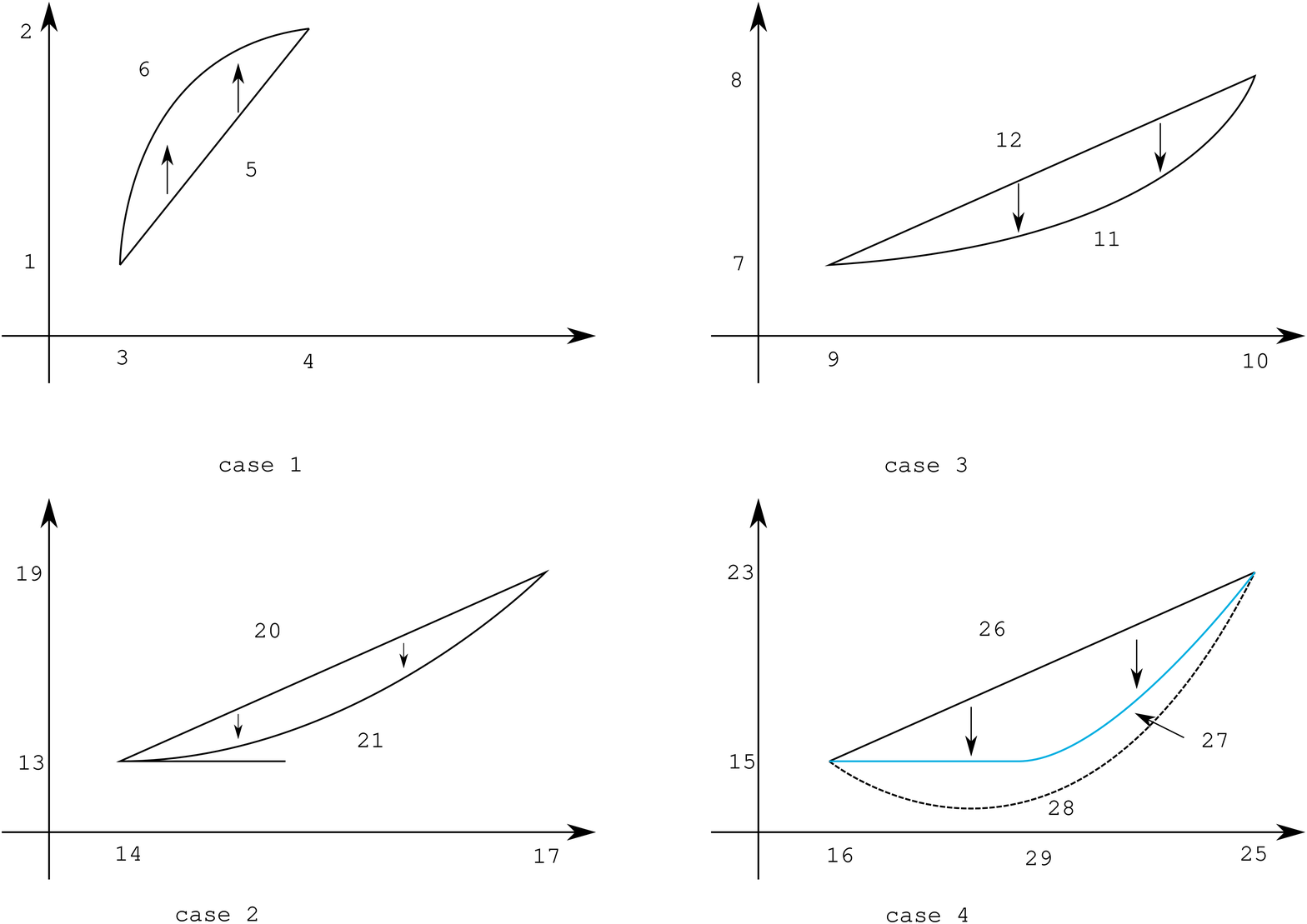}
\caption{}\label{figure2}
\end{center}
\end{figure}

\begin{proof}\ \\
Note that for the first three cases, $1<f_{\infty}(x)<\alpha$ if $1<x<\beta$. In Case 1, $f_{\infty}(x)$ is a concave function. In Cases 2 and 3, $f_{\infty}(x)$ is convex. We distinguish Case 2 and Case 3 by the fact that $f_{\infty}'(1)>0$ in Case 2 and $f_{\infty}'(1)=0$ in Case 3.

\emph{\textbf{Case 1: $\alpha > \beta$}}\ \\
By Lemma~\ref{lemma1.5}, we may choose a special initial value
\[
f_{0}(x)=\frac{\alpha-1}{\beta-1}(x-1)+1, \ x\in[1,\beta].
\]

\begin{claim}\label{claima} With conditions given as above,
$$f(x,t)-\tilde{f}(x)<0, \ \ \ \forall t\in{\mathbb{R}}, x\in (1, \beta).$$
\end{claim}
\emph{Proof of the claim:}
This is a simple application of the strong maximum principle. Since $f_{\infty}$ is the solution of (\ref{ode}) on $[1,\beta]$, thus $h(x,t):=f(x,t)-f_{\infty}(x)$ satisfies
\[
\frac{\partial h}{\partial t}= Q(f) [\frac{\partial^2 h}{\partial x}+(n-1)\frac{\frac{\partial h}{\partial x}}{x}-(n-1)\frac{h}{x^2}].
\]
Since $h_{_0}(x)=f_{_0}(x)-f_{\infty}(x) < 0$ and $h(1, t)=f(1,t)-f_{\infty}(\beta)=0$, $h(\beta, t)=f(\beta, t)-f_{\infty}(\beta)=0$, it follows from strong maximum principle that $h(x,t)<0$.

\begin{claim}\label{claimb}
$\frac{\partial f}{\partial t}\geq 0.$
\end{claim}
\emph{Proof of the claim:}  A direct computation shows that
\[
\frac{\partial f}{\partial t}|_{t=0}>0.
\]

We can prove the claim by taking time derivative of (\ref{evolutionf}), and applying the strong maximum principle.

Combining Claim~\ref{claima} and Claim~\ref{claimb}, we find that $f(x,t)$ is monotone increasing to a limiting function  $ f_{\infty}(x)=\lim_{t\to\infty}f(x,t)$ and
\begin{align} \label{2.7}
f_{\infty}(x)\leq \tilde{f}(x).
\end{align}

However, there is only one solution of the form (\ref{1.13}) satisfying (\ref{2.7}), which is exactly $\tilde{f}(x)$. Thus we have proved that $f_{\infty}(x)=\tilde{f}(x)$ for Case 1.

\emph{ \textbf{Case 2: $\alpha<\beta$,$\frac{\alpha\beta^{n-1}-1}{\beta^n-1}>\frac{n-1}{n}$}}\ \\
By Lemma~\ref{lemma1.5}, we may choose the following  initial value for the flow
\[
f_{_0}(x)=\frac{\alpha-1}{\beta-1}(x-1)+1,
\]
Similar to the Case 1, we have
\begin{claim}\label{2.61}
$$f(x,t)-\tilde{f}(x)>0, \ \ \ \forall t\in{\mathbb{R}}, x\in (1, \beta).$$
\end{claim}

\begin{claim}\label{2.71}
$$\frac{\partial f}{\partial t}\geq 0.$$
\end{claim}
For simplicity, we omit the proofs of these claims.

Since $f(x,t)$ is monotone increasing with respect to $t$, and $f_{\infty}(x)\geq \tilde{f}(x)$. $f_{\infty}$ has to be the unique solution satisfying (\ref{1.13}), which is $\tilde{f}$.  We have thus finished the proof for Case 2.\\

\emph{ \textbf{Case 3: $\alpha<\beta$,$\frac{\alpha\beta^{n-1}-1}{\beta^n-1}=\frac{n-1}{n}$}}\ \\
The proof is exactly same as that for Case 2. \\

\emph{ \textbf{Case 4: $\alpha<\beta$,$\frac{\alpha\beta^{n-1}-1}{\beta^n-1}<\frac{n-1}{n}$}}\ \\
Define \[
\tilde{g}(x) =
\begin{cases}
1, & 1\leq x\leq \lambda; \\
g(x), & \lambda \leq x\leq \beta.
\end{cases}
\]
It follows easily that Claim~\ref{2.61} and Claim~\ref{2.71} are still valid with $\tilde g$ replacing $\tilde f$. Since $f_{\infty}(x)$ satisfies (\ref{1.13}) and $f_{t}(x)\searrow f_{\infty}(x)$ as $t\to\infty$,
$$f_{\infty}(x)=\sup\{f(x)| f(x) \  {\rm satisfies \ (\ref{1.13})} \}.$$ By the characterization of $\lambda$ and $g(x)$, we obtain our conclusion.

It is easy to see that  $\lambda$ is the unique solution of
\begin{align} \label{1.14}
(n-1)\frac{\beta}{\lambda}+\frac{\lambda^{n-1}}{\beta^{n-1}}=n\alpha
\end{align} such that  $\lambda\in(1,\beta)$, and $g'(\lambda)=0$.

We have proved Theorem~\ref{111}.\end{proof}

\begin{rem}
It is interesting to remark that $\tilde{g}$ in Case 4 of Theorem~\ref{111} arises as a solution to an obstacle problem(cf.~\cite{C}).
In fact, in the convex set
\begin{align}
K:=\{ f\in H^1([1,\beta]), f(1)=1, f(\beta)=\alpha, f\geq 1 \},
\end{align}
we consider the energy functional $E: K\to \mathbb{R}$
\begin{align}
E(f)=\frac{1}{2}\int_1^{\beta} (x^{n-1}f'^2+(n-1)x^{n-3}f^2) dx.
\end{align}

The unique minimizer of $E$ in $K$ satisfying
\begin{align} \label{2.6}
\Psi_{\{f>1\}} (f''+(n-1)\frac{f'}{x}-(n-1)\frac{f}{x^2})=0.
\end{align}
\end{rem}\ \\

\begin{rem}
It worths pointing out that for all cases of Theorem~\ref{111}, $f_{\infty}'(x)$ is continuous. Further more, for Case 3 and Case 4, we have
$$f_{\infty}'(x)=0, \ \ {\rm if}\ \ f_{\infty}(x)=1.$$
This is an important feature of the limiting function that will indicate geometric properties for the geometric flow.
\end{rem}

The limiting behavior described for $f(x,t)$ can be used to determine the convergence behavior of metrics following the inverse $\sigma_{k}$ flow.

\begin{proof} [Proof of Main Theorem 1: $J$-flow case]\ We will divide the proof into four cases, just as in the proof of Theorem~\ref{111}.\\

\emph{\textbf{Case 1: $\alpha > \beta$}}\ \\
First, we claim that there exists a uniform positive lower bound for $\frac{\partial f}{\partial x}$. Suppose not, then there is sequence of $(x_n, t_n)$ such that
\[
\frac{\partial f}{\partial x}(x_n,t_n)\to 0, \quad t_n\to \infty.
\]
Then there must be an accumulation point $x_{\infty}$ in $[1,\beta]$ for a subsequence of $\{x_n\}$, which still denoted by $\{x_n\}$ for simplicity. On the other hand
\[
\lim_{n\to \infty} \frac{\partial f}{\partial x}(x_n,t_n) =f_{\infty}'(x_{\infty})\neq 0.
\]
Thus this contradiction proves the claim.
Therefore there exists a universal constant $\epsilon$ such that $\frac{\partial f}{\partial x}\geq \epsilon>0$. Consequently,  there exists a universal constant $C>0$ such that along the flow (\ref{jflow})
 \[
tr_{\omega}\chi_{\varphi}=(n-1)\frac{v'}{u'}+\frac{v''}{u''}=(n-1)\frac{x}{f}+\frac{1}{\frac{\partial f}{\partial x}}<C.
\]
Higher order estimates of $\chi_{\varphi}$ follow from Evans-Krylov and Schauder estimates. Consequently $\chi_t$   converges smoothly to $\chi_{\infty}$, which solves the critical equation (\ref{critical}). \\

\emph{ \textbf{Case 2: $\alpha<\beta$,$\frac{\alpha\beta^{n-1}-1}{\beta^n-1}>\frac{n-1}{n}$}}\ \\
Similar to case 1, since $\frac{d f_{\infty}}{d x} \geq \epsilon >0$, we may obtain a uniform positive lower bound for $\frac{\partial f}{\partial x}$.
Thus we have smooth convergence of $\chi_t$ to $\chi_{\infty}$ for the $J$-flow.\\

\emph{ \textbf{Case 3: $\alpha<\beta$,$\frac{\alpha\beta^{n-1}-1}{\beta^n-1}=\frac{n-1}{n}$}}\ \\
In this case, the previous argument fails since  $f_{\infty}'(1)=0$. A uniform positive lower bound on $\frac{\partial f}{\partial x}$ is not expected. However, we have the following

\begin{claim}\label{new} For any $\epsilon>0$, if  $ f(x,t)\geq 1+\epsilon$, there exists a positive constant $c=c(\epsilon)$ depending only on $\epsilon$, such that \[
\frac{\partial f}{\partial x}\geq c(\epsilon).
\]
\end{claim}
\begin{proof}
Suppose that the claim does not hold.
Then there exists sequence of $(x_n, t_n)$ such that
\[
\frac{\partial f}{\partial x}(x_n,t_n)\to 0, \quad t_n\to \infty.
\]
We also have that  a subsequence of $x_n$ converges to $x_{\infty}\in [1,\beta]$, for simplicity still denoted by $x_n$. Since $f(x_n,t_n)\geq 1+\epsilon$, then $f_{\infty}(x_{\infty})\geq 1+\epsilon$.
Consequently
\[
\lim_{n\to \infty} \frac{\partial f}{\partial x}(x_n,t_n) =f_{\infty}'(x_{\infty})\neq 0,
\]
we get a contradiction, thus we have proved the claim.
\end{proof}

We continue our proof of Case 3, Main Theorem 1 for $J$-flow. For any compact subset $K \subset X\setminus E_0$, there exists a constant $\epsilon>0$ such that $K\subset u'^{-1}([1+\epsilon, \alpha])$. By Claim~\ref{new},  for any $t>0$, the corresponding $f(x,t)$ defined by (\ref{cartesion}) satisfies $f_{x}(x,t)\geq c(\epsilon)$ for a given $c(\epsilon)>0$. We may conclude that  $\chi_t$   converges smoothly to $\chi_{\infty}$, following arguments given in the proof of Case 1. We have thus established the smooth convergence of the J-flow away from $E_{0}$.\\

The corresponding function
\[
v_{\infty}'(\rho)=f_{\infty}^{-1}(u_{\infty}'(\rho))
\]
is thus a potential for a K\"{a}hler metric with singularity along $E_0$.
In this case, it is easy to see that  $dd^{c} v_{\infty}$ can be extended as a smooth K\"{a}hler form to $E_{\infty}$.

On the other hand, near $\rho=-\infty$, the resulting K\"{a}hler form is not smooth. In fact,
$v_{\infty}'(\rho) \sim 1+k e^{\frac{\rho}{2}}$ near $\rho=-\infty$. In the local coordinate patch $(z_1,\cdots, z_n)$ centered at any $p\in E_0$, where $E_0 \cap U=\{z_1=0\}$, the metric is equivalent to $\frac{1}{|z_1|}dz_1\wedge d\bar{z_1} + v_{\infty}'(-\infty)\omega_{FS}$, where $\omega_{FS}$ is the Fubini-Study metric on $E_0$. Thus the metric is singular with cone angle $\pi$ transverse to $E_0$.\\

\emph{ \textbf{Case 4: $\alpha<\beta$,$\frac{\alpha\beta^{n-1}-1}{\beta^n-1}<\frac{n-1}{n}$}}\ \\
Similar to Case 3, we may prove that the flow is convergent smoothly away from $E_0$.
However, now we have
\[
\lim_{\rho\to -\infty} v'_{\infty}(\rho)=\lambda>1.
\]

The convergence
\[
v'(\rho, t)\to v'_{\infty}(\rho)
\]
and its  resulting metric flow may be understood, from a geometric point of view, as two distinct behaviors: one is a  $\delta$-concentration on $E_0$ with coefficient $(\lambda-1)$; the other is a smooth convergence of metrics away from $E_{0}$.

For $x\in(\lambda, \beta]$, it is easy to see that $v'_{\infty}$ is the corresponding limit solution for the $J$-flow of the triple $(X,\omega, \chi)$ with
\[
\omega \in \alpha[E_{\infty}]-[E_0], \quad \chi \in \beta[E_{\infty}]-\lambda[E_0].
\]

Since $g'(\lambda)=0$, this corresponds to the Case 3. One can readily check by scaling $\beta[E_{\infty}]-\lambda[E_0]$ to $\frac{\beta}{\lambda}[E_{\infty}]-[E_0]$, the condition of the Case 3 is satisfied since
\[
\frac{\alpha (\frac{\beta}{\lambda})^{n-1}-1}{(\frac{\beta}{\lambda})^n-1}=\frac{n-1}{n}.
\]
We have thus finished the proof.
\end{proof}

\section{General inverse $\sigma_k$-flow on $\mathbb{P}^n \# \overline{\mathbb{P}^n}$}

In this section, we discuss the general inverse $\sigma_k$-flow (\ref{inversek}) on $X=\mathbb{P}^n \# \overline{\mathbb{P}^n}$. We follow the discussion of Section 2, however the parabolic equation analogous to (\ref{evolutionf}) is more complicated.

\subsection{Generalized $J$-flow case ($k=1$)}\

Let $(X, \omega,\chi)$ be given as before. Assume both $\omega$ and $\chi$ safisfy Calabi Ansatz. The general inverse $\sigma_k$-flow (\ref{inversek}) for $k=1$ can be written as
\begin{align}
\frac{\partial v}{\partial t}=F((n-1)\frac{u'}{v'}+\frac{u''}{v''}).
\end{align}
Again suppose $(v'(\rho,t), u'(\rho))$ is a family of parametric curves implicitly given by
\begin{align} \label{2.3}
f(v'(\rho,t),t)=u'(\rho),
\end{align}
then the evolution of $f$ is
\begin{align} \label{2.5}
\frac{\partial f}{\partial t}=-F'Q(f) [\frac{\partial^2 f}{\partial x^2}+(n-1)\frac{\frac{\partial f}{\partial x}}{x}-(n-1)\frac{f}{x^2}]
\end{align}
By (\ref{concavity}), $F'<0$, the convergence behavior of (\ref{2.5}) is same as that of (\ref{evolutionf}).

\subsection{General case ($k>1$)}\

For general $k>1$, (\ref{inversek}) reduces to
\begin{align}
\frac{\partial v}{\partial t}=F( {n-1 \choose k-1} (\frac{u'}{v'})^{k-1} \frac{u''}{v''}+{n-1 \choose k} (\frac{u'}{v'})^k)-F(c_k),
\end{align}
then $f$ defined by (\ref{2.3}) evolves as
\begin{align}
\frac{\partial f}{\partial t}= -F' u''(\frac{f}{x})^{k-2}{n-1 \choose k-1} [\frac{\partial^2 f}{\partial x^2}\frac{f}{x}+(k-1) \frac{(\frac{\partial f}{\partial x})^2}{x}+ (n+1-2k) \frac{f \frac{\partial f}{\partial x}}{x^2} -(n-k)\frac{f^2}{x^3}].
\end{align}

Define $g(x,t):=f^{k}(x,t)$, then the evolution of $g$ is
\begin{align}\notag
\frac{\partial g}{\partial t}&=-{n-1\choose k-1}F' (\frac{f}{x})^{k-1}u''[\frac{\partial^2 g}{\partial x^2}+(n+1-2k)\frac{\frac{\partial g}{\partial x}}{x}-k(n-k)\frac{g}{x^2}]\\ \label{2.1}
&:=-F' Q(g)[\frac{\partial^2 g}{\partial x^2}+(n+1-2k)\frac{\frac{\partial g}{\partial x}}{x}-k(n-k)\frac{g}{x^2}].
\end{align}

\begin{proof} [Proof of the Main Theorem 1]

As in Section 2, we need to discuss the corresponding ODE:
\begin{align} \label{odeg}
\frac{\partial^2 g}{\partial x^2}+(n+1-2k)\frac{\frac{\partial g}{\partial x}}{x}-k(n-k)\frac{g}{x^2}=0
\end{align} with boundary values $g(1)=1$ and $g(\beta)=\alpha^k$.

One can solve it explicitly to get the solution
\begin{align} \label{2.2}
g_{\infty}(x)=ax^{k}+\frac{b}{x^{n-k}}, \quad  \quad a=\frac{\alpha^k\beta^{n-k}-1}{\beta^n-1}, \quad b=\frac{\beta^{n}-\alpha^k\beta^{n-k}}{\beta^n-1}.
\end{align}

It follows that
\[
1<g_{\infty}(x)<\alpha^k \quad \text{when}\  x \in (1,\beta),
\]
 except the case
$\alpha<\beta$ and $\frac{\alpha^k\beta^{n-k}-1}{\beta^n-1}<\frac{n-k}{n}$.

Then in that case define
\begin{align} \notag
\lambda=\inf\{ \lambda'\ | \ g\  \text{satisfies (\ref{odeg}) }, \,
  g(\lambda')=1, g(\beta)=\alpha^k \, \text{and} \, g\geq 1\}.
\end{align}

Let $g$ be the corresponding solution of (\ref{odeg}) with $g(\lambda)=1$ and $g(\beta)=\alpha^k$. Then the unique limit in this case is

\[
g_{\infty}(x) =
\begin{cases}
1, & 1\leq x\leq \lambda ;\\
g(x), & \lambda \leq x\leq \beta.
\end{cases}
\]

By the definition of $\lambda$, we have $g'(\lambda)=0$. Then it is easy to check that  $\lambda\in(1,\beta)$ is the unique  solution of
\begin{align} \label{2.4}
(n-k)(\frac{\beta}{\lambda})^{k}+k(\frac{\lambda}{\beta})^{n-k}=n\alpha^{k}.
\end{align}

Finally, notice we still have the universal constants $C_{1},C_{2}>0$ such that
\begin{align} \label{bound1}
0<C_1\leq \sigma_{k}(\underbrace{\frac{u'}{v'}, \cdots \frac{u'}{v'}}_{n-1-times}, \frac{u''}{v''})\leq C_2.
\end{align}

Since
\[
\frac{u'}{v'}\geq \frac{1}{\beta},
\]
 from (\ref{bound1}) we get
a uniform upper bound for $\frac{u''}{v''}$.

The rest of the proof follows that of Section 2.
\end{proof}

\section{Flows on $\mathbb{P}(\mathcal{O}_{\mathbb{P}^n}\oplus \mathcal{O}_{\mathbb{P}^n}(-1)^{\oplus (m+1)})$}

In this section, we consider the general inverse $\sigma_k$-flow on a family of projective bundles over $\mathbb{P}^n$ and its convergence behavior under the assumption both $\omega$, $\chi$ satisfy the Calabi Ansatz. A new geometric limit phenomenon occurs.

Let $E=\mathcal{O}_{\mathbb{P}^n}\oplus \mathcal{O}_{\mathbb{P}^n}(-1)^{\oplus (m+1)}$ be a vector bundle over a projective space $\mathbb{P}^n$, where $\mathcal{O}_{\mathbb{P}^n}$ is the trivial line bundle and $\mathbb{O}_{\mathbb{P}^n}(-1)$ is the tautological line bundle. Let
\[
X_{m,n}=\mathbb{P}(\mathcal{O}_{\mathbb{P}^n}\oplus \mathcal{O}_{\mathbb{P}^n}(-1)^{\oplus (m+1)})
\]
be the projectivization of $E$. $X_{m,n}$ is a $\mathbb{P}^{m+1}$ bundle over $\mathbb{P}^n$ with  $\pi:X_{m,n}\to \mathbb{P}^n$ being the bundle map.  In particular, $X_{0,n}$ is $\mathbb{P}^{n+1}$ blown up at one point. Let $D_{\infty}$ be the divisor in $X_{n,m}$ given by $\mathbb{P}(\mathcal{O}_{\mathbb{P}^n}(-1)^{\oplus (m+1)})$ and $D_0$ be the divisor in $X_{m,n}$ given by $\mathbb{P}(\mathcal{O}_{\mathbb{P}^n}\oplus \mathcal{O}_{\mathbb{P}^n}(-1)^{\oplus m})$. In fact, the additive divisor group $N^1(X_{m,n})$ is spanned by $[D_0]$ and $[D_{\infty}]$. We also define the divisor $D_{H}$ by the pullback of the divisor on $\mathbb{P}^n$ associated to $\mathcal{O}_{\mathbb{P}^n}(1)$. Then
\[
[D_{\infty}]=[D_0]+[D_{H}].
\] Moveover, $D_{\infty}$ is a big and semi-ample divisor and any divisor $a[D_H]+b[D_{\infty}]$ is ample if and only if $a>0$ and $b>0$.

To consider the Calabi Ansatz (See ~\cite{C, SY}), let $\omega_{FS}$ be the Fubini-Study metric on $\mathbb{P}^n$. Let $h$ be the hermitian metric on $\mathcal{O}_{\mathbb{P}^n}(-1)$ such that $Ric(h)=-\omega_{FS}$. Under local trivialization of $E$, we write
\[
e^{\rho}=h(z) |\xi|^2, \xi=(\xi_1, \xi_2, \cdots, \xi_{m+1}),
\] where $h(z)$ is a local representation of $h$. In particular, if we choose an inhomogeneous coordinate $z=(z_1, z_2, \cdots, z_n)$ on $\mathbb{P}^n$, we have
\[
h(z)=1+|z|^2.
\]
We consider K\"{a}hler metrics of following type on $X_{m,n}$:
\begin{align} \label{calabi}
\omega=a\pi^{*}\omega_{FS}+\frac{\sqrt{-1}}{2\pi} \partial \bar{\partial} u(\rho).
\end{align}

According to Calabi~\cite{Ca},  (\ref{calabi}) is K\"{a}hler if and only if
 $a>0$, $u'>0$, $u''>0$, and asymptotic behavior of $u$ satisfies:
 \begin{align} \label{criterion}
 (\dagger).& \quad u_0(r):=u(\ln r) \\\notag
 &\quad \text{is extendable by continuity to a smooth function at} \, r=0,\text{and}\, u_0'(0)>0.\\ \notag
 (\ddagger).& \quad u_{\infty}(r):=u(-\ln r)+b\ln r\quad \\ \notag
 &\text{ is extendable by continuity to a smooth function at} \, r=0\, \text{for some $b$},\text{and}\, u_{\infty}'(0)>0.
 \end{align}

Thus we have $\lim_{\rho\to -\infty} u(\rho)=0$ and $\lim_{\rho \to \infty}u(\rho)=b$. Here $\rho=-\infty$ corresponds $P_0$ and $\rho=\infty$ corresponds to $D_{\infty}$. Furthermore,
\begin{align}
\omega \in a[D_H]+b[D_{\infty}].
\end{align}
Note
\begin{align} \label{3.2}
\omega=(a+u')\omega_{FS}+\frac{\sqrt{-1}}{2\pi} he^{-\rho}(u'\delta_{ij}+he^{-\rho}(u''-u')\bar{\xi_i}\xi_{j}) \nabla\xi_i \wedge \nabla \bar{\xi_j},
\end{align}
where $\nabla \xi_i = d \xi_i +h^{-1}\partial h \xi_i.$

We may now discuss the general inverse $\sigma_k$-flow (\ref{inversek}) on $X_{m,n}$. If $\omega$, $\chi$ are of the form (\ref{calabi}), without loss of generality, we may normalize them so that
\begin{align}
\omega &\in [D_H]+b[D_{\infty}], b>0.\\ \notag
\chi &\in [D_H]+b'[D_{\infty}], b'>0.
\end{align}
Hence we can assume that
\begin{align}
\omega=\omega_{FS}+ b \partial \bar{\partial} u(\rho), \quad \chi=\omega_{FS}+b' \partial \bar{\partial} v(\rho),
\end{align}
with $u$, $v$ satisfying the criterion (\ref{criterion}). The general inverse $\sigma_k$-flow preserves the Calabi Ansatz (\ref{calabi}). We consider the function $f(x,t)$ determined by
\begin{align} \label{3.1}
f(v'(\rho,t),t)=u'(\rho).
\end{align}

\begin{prop}
Consider the general inverse $\sigma_k$-flow on the triple $(X_{m,n}, \omega, \chi)$. If $\omega$, $\chi$ are given as above, then the evolution of $f(x,t)$ defined via (\ref{3.1}) is
\begin{align} \label{evolution}
\frac{\partial f}{\partial t}=-F'u'' (\sigma_k(\underbrace{\frac{1+f}{1+x}}_{n-copies}, \underbrace{\frac{f}{x}}_{m-copies}, f'))'.
\end{align}

\end{prop}

\begin{proof} From (\ref{3.2}), we may calculate  the eigenvalues of $\chi$ with respect to $\omega$ to be:
\begin{align}
\underbrace{\frac{1+v'}{1+u'}, \cdots, \frac{1+v'}{1+u'}}_{n-times}, \underbrace{\frac{v'}{u'}, \cdots, \frac{v'}{u'}}_{m-times}, \frac{v''}{u''}.
\end{align}

Taking time derivative of (\ref{3.1}), we get
\begin{align} \label{3.7}
\frac{\partial f}{\partial t}=- f' \frac{\partial v'}{\partial t}.
\end{align}
Taking spacial derivative of (\ref{inversek}), we get
\begin{align} \label{3.3}
\frac{\partial v'}{\partial t}= F' (\sigma_k(\underbrace{\frac{1+u'}{1+v'}}_{n-copies}, \underbrace{\frac{u'}{v'}}_{m-copies}, \frac{u''}{v''}))'
\end{align}
We also have
\begin{align} \label{3.4}
f'(\frac{1+u'}{1+v'})'&= f'[\frac{u''}{1+v'}-\frac{(a+u')v''}{(a+v')^2}]\\ \notag
                      &=u'' [\frac{f'}{1+v'}-\frac{1+u'}{(1+v')^2}]= u'' (\frac{1+f}{1+x})'.
\end{align}
Similarly,
\begin{align} \label{3.5}
f'(\frac{u'}{v'})'=u''(\frac{f}{x})'.
\end{align}
We also have
\begin{align} \label{3.6}
f'(\frac{u''}{v''})'&= f'[\frac{u'''}{v''}-\frac{u''v'''}{v''^2}] \\ \notag
&=f'[\frac{f''v''^2+f'v'''}{v''}-\frac{u''v'''}{v''^2}] =u''f''.
\end{align}

Using (\ref{3.3}), (\ref{3.4}), (\ref{3.5}) and (\ref{3.6}) and (\ref{3.7}), we get
\begin{align} \label{evolution2}
\frac{\partial f}{\partial t}=-F'u'' (\sigma_k(\underbrace{\frac{1+f}{1+x}}_{n-copies}, \underbrace{\frac{f}{x}}_{m-copies}, f'))'.
\end{align}
\end{proof}

Following our method developed in Section 2, we consider  the ODE
\begin{align} \label{3.8}
(\sigma_k(\underbrace{\frac{1+f}{1+x}}_{n-copies}, \underbrace{\frac{f}{x}}_{m-copies}, f'))'=0,\ x\in[0,b'],
\end{align}
with the boundary condition $f(0)=0$ and $f(b')=b$.

\begin{prop}
The equation (\ref{3.8}) admits a unique positive monotone increasing solution with $f(0)=0$ and $f(b')=b$ if and only if
\begin{itemize}
\item  $c_k \geq {n\choose k}$ when $k\leq n$, and the corresponding unique solution satisfies $f'(0)>0$ if and only if the strict inequality holds.
\item $c_k >0$ when $k>n$ and the corresponding unique solution always satisfies $f'(0)>0$.
\end{itemize}
Here $c_k$ is the topological constant
\[
\frac{{n+m+1 \choose k}\int_{X_{m,n}} \chi^{m+n+1-k}\wedge \omega^k}{\int_{X_{m,n}} \chi^{m+n+1}}.
\]
\end{prop}

\begin{proof}
Define
\begin{align} \label{G}
G^{m,n}_{a}(x)=\int_{0}^{x} t^m(t+a)^n dt,
\end{align}
then
\begin{align} \label{3.9}
\frac{d [G^{m,n}_{1+t} (x+tf)]}{dx}&= (x+tf)^m(1+t+x+tf)^n(1+ tf') \\ \notag
                                   &= x^m(1+x)^n(1+t \frac{f}{x})^m(1+ t\frac{f+1}{x+1})^n (1+t f') \\ \notag
                                   &= [G^{m,n}_1(x)]'[1+ t\sigma_1(\underbrace{\frac{1+f}{1+x}}_{n-copies}, \underbrace{\frac{f}{x}}_{m-copies}, f')\\ \notag
                                   &+ t^2\sigma_2(\underbrace{\frac{1+f}{1+x}}_{n-copies}, \underbrace{\frac{f}{x}}_{m-copies}, f')+ \cdots \\ \notag
                                   &\cdots + t^{m+n+1} \sigma_{m+n+1}(\underbrace{\frac{1+f}{1+x}}_{n-copies}, \underbrace{\frac{f}{x}}_{m-copies}, f')].
\end{align}
Let
\[
 G^{m,n,k}_{1}(f,x)=\frac{1}{k!}\frac{d^k}{d t^k}|_{t=0} (G^{m,n}_{1+t}(x+tf)),
\]
taking $k$-th $t$ derivative of (\ref{3.9}) and evaluating at $t=0$, we find ODE (\ref{3.8}) is equivalent to
\begin{align} \label{3.10}
G^{m,n,k}_{1}(f,x)= \alpha G^{m,n}_1(x) +\beta,
\end{align}
for two constants $\alpha$ and $\beta$.

\begin{claim}
\begin{align} \label{3.11}
G^{m,n,k}_1(f,x)=f^ka_k(x)+f^{k-1}a_{k-1}(x)+\cdots +f a_1(x) +a_0(x),
\end{align}
where $a_i(x)$ are polynomials of $x$. In particular,  $a_0(x)={ n\choose k} G^{m,n-k}_1(x)$ when $k\leq n$ and $a_0=0$ when $k>n$.
\end{claim}
\begin{proof}The proof is a straightforward computation based on explicit form of $G^{m,n}_1$ via (\ref{G}).
\end{proof}
From (\ref{3.10}) and (\ref{3.11}), the boundary value of $f$ implies
\[
 \alpha=\frac{G^{m,n,k}_1(b, b')}{G^{m,n}_1(b')} \quad \text{and} \quad \beta=0.
\]

Notice that $\frac{G^{m,n,k}_1(b, b')}{G^{m,n}_1(b')}$ is actually the topological constant $c_k$. This follows from direct computation using metrics $\omega$ and $\chi$ of the form given here. It is also easy to verify that all coefficients $a_i(x)$ are polynomials with positive coefficients. Thus for each fixed $x$ view $G^{m,n,k}_1(f,x)$ as a polynomial of $f$, it admits a unique positive solution if and only if
\[
\alpha G^{m,n}_1(x)\geq a_0(x),
\] with equality holds only at $x=0$.
By definition,
\begin{align}
\alpha G^{m,n}_1(x)-a_0(x)=\int_{0}^x t^m (t+1)^{n-k}(\alpha (t+1)^k-{n \choose k}) dt,
\end{align}
from which $c_k \geq {n\choose k}$ follows.

It is clear that $f'(x)\geq 0$. We claim that the function $f$ is strictly increasing. If not, there exists a $x_0\in [0,b']$ such that $f'(x_0)=0$. At $x_0$, we have
\begin{align}
(\frac{f}{x})'(x_0)=\frac{-f(x_{0})}{x_{0}^2}<0, \quad \text{and} \quad(\frac{1+f(x)}{1+x})'(x_0)=\frac{-f(x_{0})-1}{(1+x_{0})^2}<0.
\end{align}
Both $\frac{f}{x}$ and $\frac{1+f}{1+x}$ are strictly decreasing near $x_{0}$. Note that
\[
\sigma_k(\underbrace{\frac{1+f}{1+x}}_{n-copies}, \underbrace{\frac{f}{x}}_{m-copies}, f')=\alpha.
\]
Therefore, $f''(x_{0})>0$. Hence  the points where
$f'=0$ are discrete. It follows that $f$ is strictly increasing.

Finally, to calculate $f'(0)$, we  expand (\ref{3.10}) near $x=0$ and compare the lowest order terms of both sides.
Assume
\[
f(x)=Ax+\text{higher order terms},
\]
We may derive that
\begin{align}
\frac{1}{k!} \frac{d^k}{dt^k}|_{t=0}\frac{(x+tf)^{m+1}(1+t)^n}{m+1}= \frac{{n\choose k}}{m+1}x^{m+1} + \text{ higher order terms}.
\end{align}
Hence $A=f'(0)>0$ if and only if $\alpha>{n \choose k}$ when $k\leq n$ and $\alpha>0$ when $k>n$.
\end{proof}

If $c_k <{n \choose k}$, the solution of (\ref{3.8}) with the boundary condition $f(0)=0$ and $f(b')=b$ is not positive near $x=0$. We define
\begin{align} \notag
\lambda:=\inf\{ \lambda'\ | \text{\ $f(x)$ solves (\ref{3.8}) s.t.} \,
  f(\lambda')=0, f(b')=b\, \text{and} \, f\geq 0\}
\end{align} and let $f$ be the corresponding solution of (\ref{3.8}) with $f(\lambda)=0$ and $f(b')=b$. By the definition of $\lambda$, $f'(\lambda)=0$.

\begin{prop}\label{rrr}
There exists a unique parameter-triple $(\alpha, \beta, \lambda)$, all positive with $\lambda \in (0,b')$, such that
\begin{align} \label{3.14}
G^{m,n,k}_1(f,x)=\alpha G^{m,n}_1(x)+\beta
\end{align}
defines implicitly a unique positive monotone increasing solution of (\ref{3.8}) on $[\lambda, b']$, satisfying
\begin{align} \label{3.12}
f(\lambda)=0, \quad f(b')=b, \quad f'(\lambda)=0.
\end{align}
\end{prop}

\begin{proof}
By (\ref{3.11}), $f(x)$ satisfies (\ref{3.12}) if and only if
\begin{align}\label{3.13}
\left\{
\begin{array}{l l}
a_0(\lambda) &=
\alpha G^{m,n}_1(\lambda)+\beta,\\
a_0'(\lambda)&=\alpha (G^{m,n}_1)'(\lambda), \\
G^{m,n,k}_1(b,b')&=\alpha G^{m,n}_1(b')+\beta.
\end{array}
\right.
\end{align}

It is easy to see there is a unique solution $(\alpha, \beta, \lambda)$ to (\ref{3.13}).
\end{proof}

Following the discussion in Section 2 and 3, we conclude that the limit of $f_{t}$ in this case is

\begin{align} \label{3.17}
f_{\infty}(x) =
\begin{cases}
0, & 0\leq x\leq \lambda \\
f(x), & \lambda \leq x\leq b',
\end{cases}
\end{align}
where $f(x)$ is defined in Proposition~\ref{rrr}.

Now we are in position to prove Main Theorem 2.

\begin{proof} [Proof of Main Theorem 2]\

For simplicity, we denote $\frac{\partial f(x,t)}{\partial x}$ by $f'_{t}$ and $f(\cdot, t)$ by $f_t$.
We first prove the following

\begin{claim}\label{ppp}
There is a universal constant $C$ depending only on initial values, such that
$$f'_{t}\leq C.$$
\end{claim}

\begin{proof}
The two-sided bound (\ref{1.8}) leads to
\begin{align} \label{bound2}
C_1\leq \sigma_k(\underbrace{\frac{1+f_t}{1+x}}_{n-copies}, \underbrace{\frac{f_t}{x}}_{m-copies}, f'_t) \leq C_2.
\end{align}

We separate the proof into following two cases.\\

\emph{Case 1: $k\leq n$} \ \\
In this case, there is always a term $(\frac{1+f_t}{1+x})^{k-1}f'_t$ in $\sigma_k$, hence from (\ref{bound2}) we have
\[
(\frac{1+f_t}{1+x})^{k-1}f'_t\leq C_2.
\]
Since $f_t$ takes value in $[0,b]$, the term $\frac{1+f_t}{1+x}$ is bounded from below, from which a uniform upper for $f'_t$ follows.\\

\emph{Case 2: $k>n$} \
\begin{claim}\label{qqq}
$\frac{f_t(x)}{x}\geq C(C_1, k, m, n, b, b')>0$ for some universal constant $C$ depending only on $C_1, k, m, n, b$ and $b'$.
\end{claim}
\emph{Proof of Claim~\ref{qqq}: } For a fixed $t$, let $x_0$ be the point where $\frac{f_t(x)}{x}$ achieves its minimum. Such point exists since $\frac{f_t(x)}{x}$ is a continuous function on $[0,b']$.
At $x_0$, we have
\begin{align} \label{3.16}
f'_{t}(x_0)=\frac{f_{t}(x_0)}{x_0}.
\end{align}
Indeed, if $x_n=0$, then (\ref{3.16}) is trivially true by the fact $f_{t}(0)=0$. If $x_n$ is an interior point, then
\[
(\frac{f_{t}(x)}{x})'=\frac{f'_{t}-\frac{f}{x}}{x}.
\]
(\ref{3.16}) follows as well. If $x_0=b'$, then $\frac{f_{t}(b')}{b'}=\frac{b}{b'}$. We are done for the lower bound of $\frac{f_t(x)}{x}$. Concerning the lower bound, we can also assume that $\frac{f_{t}(x_0)}{x_0}\leq 1$. Then $\frac{1+f_{t}}{1+x}\leq 1$ as well.
Take $x=x_0$ in $\sigma_k(\underbrace{\frac{1+f_t}{1+x}}_{n-copies}, \underbrace{\frac{f_t}{x}}_{m-copies}, f'_t)$, by (\ref{3.16}) and lower bound in (\ref{bound2}), we get a uniform lower bound $C$ on $\frac{f_t(x)}{x}$ depending only on $C_1, k, m, n, b, b'$ not on $t$. We have finished the proof of Claim~\ref{qqq}.\\

We continue our proof of Claim~\ref{ppp}, Case 2. By Claim~\ref{ppp}, both $\frac{f_t(x)}{x}$ and $\frac{1+f_t(x)}{1+x}$ are bounded uniformly from below, using the upper bound in (\ref{bound2}) again, we get a uniform upper bound for $f'_t$.

We have thus finished the proof of Claim~\ref{ppp}.
\end{proof}
One more thing to mention is the handle of nonlinearity in this situation. As a matter of fact we already have seen that
\begin{align}
\sigma_k(\underbrace{\frac{1+f_t}{1+x}}_{n-copies}, \underbrace{\frac{f_t}{x}}_{m-copies}, f'_t)=\frac{G^{m,n,k}_{1}(f,x)'}{G^{m,n}_1(x)'},
\end{align}
thus the parabolic equation (\ref{evolution2}) on $f$ can be written as
\begin{align}
\frac{\partial G^{m,n,k}_{1}(f,x)}{\partial t}&= \frac{\partial G^{m,n,k}_{1}(f,x)}{\partial f} \frac{\partial f}{\partial t} \\ \notag
                                              &= \frac{\partial G^{m,n,k}_{1}(f,x)}{\partial f} (\frac{G^{m,n,k}_{1}(f,x)'}{G^{m,n}_1(x)'})' \\ \notag
                                              &= Q(f,x) (\frac{G^{m,n,k}_{1}(f,x)''}{G^{m,n}_1(x)'}-\frac{G^{m,n,k}_{1}(f,x)'G^{m,n}_1(x)''}{(G^{m,n}_1(x)')^2}),
\end{align}
which becomes a degenerate parabolic equation for $G^{m,n,k}_{1}(f,x)$. This is a generalization of treatment in Section 3 for the general $k$ case. The uniform upper bound on $\frac{\partial f}{\partial x}$ implies the uniform upper bound on $G^{m,n,k}_{1}(f,x)'$. The rest part of the convergence is similar. We omit it for simplicity.
Once we get the convergence
\[
\lim_{t\to \infty} G^{m,n,k}_{1}(f,x)=G^{m,n,k}_{1}(f_{\infty},x),
\]
we could infer that
\[
\lim_{t\to \infty} f(x,t)=f_{\infty}(x),
\] since $G^{m,n,k}_{1}(f,x)$ is monotone increasing on the $f$ variable.\\

Finally, let us discuss the geometric behavior.

First of all, if $f_{\infty}'(0)>0$, which is the case for $k>n$ and $c_k>{n \choose k} $ when $k\leq n$, then as previously discussed, we get smooth convergence.

If $f_{\infty}'(0)=0$, we have smooth convergence away from $\{\rho=-\infty\}$ which corresponds to $P_0$. Then the corresponding K\"{a}hler metric $dd^{c}v_{\infty}$ has a conical singularity of cone angle $\pi$ transverse to $P_0$. $P_0$ can be also regarded as the intersection of $m+1$ effective divisors $\mathbb{P}(\mathcal{O}_{\mathbb{P}^n}\oplus \mathcal{O}_{\mathbb{P}^n}(-1)^{\oplus m})$, which is of codimension $m+1$. Therefore the convergence of general inverse $\sigma_k$-flow can produces K\"{a}hler metrics which are singular on subvarieties of higher codimension.

If $c_k<{n\choose k}$, the limit is given by (\ref{3.17}). We can still obtain smooth convergence away from $P_0$. The limit
\[
\lim_{\rho \to -\infty}v'_{\infty}(\rho) =\lambda \neq 0
\]
corresponds to a blow up along $P_0$ in the following sense which we explain.

Away from $P_0$, we have the smooth convergence
 \[
 (X_{m,n}\setminus P_0, \chi_t) \to  (X_{m,n}\setminus P_0, \chi_{\infty}:=dd^{c} v_{\infty})
  \]
as $t\to \infty$. We consider the metric completion of $(X_{m,n}\setminus P_0, \chi_{\infty})$.  If we restrict $dd^{c}v_{\infty}$ fiber-wise, we may find its metric completion is homeomorphic to the blow up of $\mathbb{P}^{n+1}$ at one point $\rho=-\infty$. Thus globally, the metric completion of $(X_{m,n}\setminus P_0, \chi_{\infty})$ is homeomorphic to $\tilde{X}_{m,n}$, the blow up of $X_{m,n}$ along $P_0$.  Let $\pi: \tilde{X}_{m,n}\to X_{m,n}$ be the blow up map and let $E$ be the exceptional divisor, which is homeomorphic to $\mathbb{P}^n\times \mathbb{P}^m$.  Since $f_{\infty}'(\lambda)=0$, it follows that the pull back metric $\pi^{*}(dd^c v_{\infty})$ on $\tilde{X}_{m,n}$ is a K\"{a}hler metric with conical singularity of angle $\pi$ transverse to the fibre direction of $E$. Moreover, since $\lim_{\rho \to -\infty}v'_{\infty}(\rho) =\lambda$, we have
\[
\pi^{*}(dd^c v_{\infty})\in [D_H]+b'[D_{\infty}]-\lambda[E].
\]

Since $f_{\infty}$ satisfies
\[
G^{m,n,k}_1(f,x)=\alpha G^{m,n}_1(x)+\beta,
\]
for some constants $\alpha, \beta$ and $x\in[\lambda, b']$, hence we have equation
\begin{align}
\alpha \pi^{*}(\chi_{\infty})^{m+n+1}= \pi^{*}(\chi_{\infty})^{m+n+1-k}\wedge \pi^{*}(\omega)^k
\end{align}
holds on $\tilde{X}_{m,n}$ away from $E$.  $\alpha$ is the corresponding topological constant
\[
\alpha=\frac{[\chi]^{m+n+1-k}[\omega]^k}{[\chi]^{m+n+1}}
\] with $[\chi]\in b'[\pi^{*}(D_{\infty})]-\lambda[E]+[\pi^{*}(D_H)]$ and $[\omega]\in b[\pi^{*}(D_{\infty})]+[\pi^{*}(D_H)]$.
\end{proof}

\begin{rem}
Motivated by the last case of Main Theorem 2, we can also study the general inverse $\sigma_k$-flow on $\tilde{X}_{m,n}$. $N^1(\tilde{X}_{m,n})$ is spanned by
$[D_H]$, $[D_{\infty}]$ and $[E]$, where we use the same notation $[D_H]$ and $[D_{\infty}]$ to denote the pull back of corresponding divisors on $X_{m,n}$.  The class
\[
p[D_{\infty}]-q[E]+r[D_H]
\] is K\"{a}hler if and only if $p>q>0$ and $r>0$.
On a local coordinates $(z_1,\cdots z_n, \xi_0, \cdots \xi_m)=\mathbb{C}^n\times (\mathbb{C}^{m+1}\setminus \{0\})$, let
\[
e^{\rho}=h(z)|\xi|^2,
\] where $h(z)=1+|z_1|^2+|z_2|^2+\cdots+|z_n|^2$.

We consider K\"{a}hler metrics of the form
\[
r\omega_{FS}+\frac{\sqrt{-1}}{2\pi}\partial \bar{\partial} u(\rho),
\]
with $u$ satisfies proper asymptotic behavior (\ref{asymptotic}) near $\rho=-\infty$ and $\rho=\infty$. The convergence behavior is very similar to that on $\mathbb{P}^n\#\overline{\mathbb{P}^n}$ so we will omit the detail here. \end{rem}

\end{document}